\documentclass[10pt,letterpaper,leqno]{article}
\usepackage{amssymb}
\usepackage{amsmath}
\usepackage{amsrefs}
\usepackage{amsthm}
\usepackage{array}
\usepackage{rotating}
\usepackage{verbatim}
\usepackage[matrix,tips,frame,color,line,poly]{xy}
 
\theoremstyle{plain}
\newtheorem{theorem}[equation]{Theorem}
\newtheorem{corollary}[equation]{Corollary}
\newtheorem{lemma}[equation]{Lemma}

\newtheorem{question}[equation]{Question}
\newtheorem{proposition}[equation]{Proposition}

\theoremstyle{definition}

\theoremstyle{indenteddefinition}

\newtheorem{example}[equation]{Example}
\newtheorem{remark}[equation]{Remark}

\newtheoremstyle{named}{}{}{}{}{\bfseries}{}{.5em}{#3}
\theoremstyle{named}
\newtheorem*{namedtheorem}{Theorem}

\renewcommand{\sec}[1]{\section{#1}
\renewcommand{\theequation}{\thesection.\arabic{equation}}
  \setcounter{equation}{0}}
\newcommand{\subsec}[1]{\subsection{#1}
\renewcommand{\theequation}{\thesubsection.\arabic{equation}}
  \setcounter{equation}{0}}

\DeclareMathOperator\Norm{Norm}

\newcommand{\Vreg}{V_{\text{reg}}}
\newcommand{\Hom}{\text{Hom}}
\newcommand{\inv}{^{-1}}
\newcommand\Q{\mathbb Q}
\newcommand\R{\mathbb R}
\newcommand\C{\mathbb C}
\newcommand\Z{\mathbb Z}

\newcommand{\ch}[1]{#1^\vee}
\newcommand\tits{\mathcal T}
\newcommand\Cal{\mathcal C}
\renewcommand\P{\mathcal P}
\newcommand\E{\mathcal E}
\renewcommand\o{\tilde o}
\newcommand\oad{\tilde o_{\text{ad}}}
\newcommand\lcm{\text{LCM}}
\newcommand\rank{\text{rank}}
\newcommand\tchar{\text{char}}  
\newcommand\Semispin{\text{Semispin}}
\newcommand\Spin{\text{Spin}}
\newcommand\SO{\text{SO}}
\newcommand\PSO{\text{PSO}}
\newcommand\SL{\text{SL}}
\newcommand\GL{\text{GL}}
\newcommand\diag{\text{diag}}
\newcommand\Cox{\text{Cox}}
\newcommand{\deltaW}{\negthinspace\negthinspace\phantom{a}^\delta W}
\newcommand{\deltaG}{\negthinspace\negthinspace\phantom{a}^\delta G}
\newcommand{\deltaT}{\negthinspace\negthinspace\phantom{a}^\delta\tits}
\newcommand{\deltaN}{\negthinspace\negthinspace\phantom{a}^\delta N}

\newcommand{\gti}{\negthinspace\negthinspace\phantom{a}^tg^{-1}}

\newcommand{\twoE}{\negthinspace\negthinspace\phantom{a}^2 E_6}

\newcommand{\threeD}{\negthinspace\negthinspace\phantom{a}^3 D_4}

\newcommand{\Gad}{G_{\text{ad}}}

\def\d{\delta}

\newcommand\wt{\widetilde}

\renewcommand{\theequation}{\thesection.\arabic{equation}}

\begin{document}
\title{Lifting of elements of Weyl groups}
\author{Jeffrey  Adams\thanks{The first author was supported in part
    by NSF grant DMS-1317523.} \\Department of
  Mathematics \\ University of Maryland
\and Xuhua He\thanks{The second author was supported in part
    by NSF grant DMS-1463852.} \\Department of
  Mathematics \\ University of Maryland
}

\date{\today}
\maketitle

\section{Introduction}\label{s:intro}

Let $G$ be a connected reductive group over an algebraically closed
field $F$. Choose a Cartan subgroup $T\subset G$, let $N=\Norm_G(T)$ be
its normalizer, and let $W=N/T$ be the Weyl group.
We have the exact sequence 
\begin{equation}
\label{e:exactW}
1\rightarrow T\rightarrow N\overset p\rightarrow W\rightarrow 1.
\end{equation}
It is natural to ask what can be said about
the orders of lifts of an element $w\in W$ to $N$. What is the smallest
possible order of a lift of $w$?  In particular, can $w$ be lifted to an element of $N$ of the same order?

Write $o(*)$ for the order of an element of a group, and let $N_w=p\inv(w)\subset N$. 
Define
\begin{subequations}
\renewcommand{\theequation}{\theparentequation)(\alph{equation}}  
\begin{equation}
\label{def:dtilde}
\o(w,G)=\min_{g\in N_w} o(g).
\end{equation}
The most important case is for the adjoint group $\Gad$, so define
\begin{equation}
\oad(w)=\o(w,\Gad).
\end{equation}
It is clear that $\o(w,G)$ only depends on the conjugacy class $\Cal$ of $w$, 
so write $\o(\Cal,G)$ and $\oad(\Cal)$ accordingly.
\end{subequations}

An essential role is played by the Tits group. 
This is a group which fits in an exact sequence
$1\rightarrow T_2\rightarrow \tits\rightarrow W\rightarrow 1$ where $T_2$ is a certain subgroup of 
the elements of $T$ of order ($1$ or) $2$. 
This implies $\o(w,G)=o(w)$ or $2o(w)$, but it can be difficult to determine which case holds. 

We also consider the twisted situation.
Let $\delta$ be an automorphism of $G$ of finite order which
preserves a pinning, and set  $\deltaG=G\rtimes\langle\delta\rangle$.
Let $\deltaN=\Norm_{\deltaG}(T)$ and $\deltaW=\deltaN/T$. 
Then conjugacy  in $W\delta$ is the same as $\delta$-twisted conjugacy in $W$, 
and we can ask about the order of lifts of elements of $W\delta$ to $\deltaN$.
See Section \ref{s:tits} for details. 

We say $W$ lifts to $G$ if the exact sequence \eqref{e:exactW} splits, 
in which case  $\o(w,G)=\o(w)$ for
all $w$.  If this is not the case, it may  not be practical to
give a formula for $\o(w,G)$ for all conjugacy classes. 
Rather, this can be done for 
several natural families.
We say  $w\in W\delta$ is elliptic if it has no nontrivial fixed vectors 
in the reflection representation; in this case 
all lifts of $w$ 
are conjugate, so have the same order
$\o(w,G)$. An element $w\in W\delta$ is said to be regular if it has a regular eigenvector (see Section \ref{s:regular} and \cite{springer_regular}).

Let $\ch\rho$ be one-half the sum of the positive coroots in any
positive system.  We refer to the element $z_G=(2\ch\rho)(-1)$ as the
{\it principal involution} in $G$.  It is contained in the center
$Z(G)$, is independent of the choice of positive system, and is fixed
by every automorphism of $G$.

Here is a result concerning when $W$ lifts, so $\o(w)=o(w)$ for all $w$.

\begin{namedtheorem}[Theorem A]
\label{t:WsplitIntro}  
If the characteristic of $F$ is $2$, then the Tits group $\tits$ is isomorphic to the Weyl group, so the 
exact sequence \eqref{e:exactW} splits.

Suppose the characteristic of $F$ is not $2$, and that $G$ is simple.\footnote{By simple we mean in the sense of algebraic groups:
$G$ has no nontrivial,  closed, connected, normal subgroups. Some authors use the term \emph{quasi-simple} or \emph{almost simple}.
}
If $G$ is adjoint of type $A_n,B_n,D_n$ or $G_2$ then $W$ lifts. 
The same holds for $\SO(2n)$, and in type $A_n$ if $|Z(G)|$ is odd.
For  necessary and sufficient conditions for $W$ to lifts see Theorem \ref{t:Wsplit}. 
\end{namedtheorem}
Over $\C$ this is proved in \cite{Weyl_splitting}, with the exception of some cases in types $A_n$ and $D_n$.

\begin{namedtheorem}[Theorem B]
\label{t:A}
Assume the characteristic of $F$ is not $2$. 
\begin{enumerate}
\item
Suppose $G$ is simple and $w\in W\delta$ is an elliptic element.
Then $\oad(w)=o(w)$,
except in certain cases in type $C_n$, or  
$G$ is of  type $F_4$ and $w$ is in the conjugacy class
$A_3+\tilde A_1$. See Section \ref{s:good} for details.
\item If $w\in W\delta$ is regular then $\oad(w)=o(w)$.
\end{enumerate}
\end{namedtheorem}
The case when $w$ is regular and elliptic is discussed in 
\cite{rgly}.
The next  result gives more detail on $\o(w)$ for elliptic conjugacy classes. 
 
\begin{namedtheorem}[Theorem C]
\label{t:B}
Assume the characteristic of $F$ is not $2$. 
Suppose $G$ is simple,  $w$ is an elliptic element of $W\delta$, and $g$ is a lift of $w$. 

\begin{enumerate}
\item Suppose $G$ is of type $A_{n}$. Then $g^{o(w)}=z_G$.
\item Suppose $G$ is of type $C_n$. If $G$ is simply connected
then $g^{o(w)}\ne 1$. 
The elliptic conjugacy classes are parametrized by partitions of $n$ (cf. Section \ref{s:ellclassical}).
Suppose $G$ is adjoint and $w$ corresponds to a partition $(a_1,\dots, a_l)$.
Then $g^{o(w)}=1$ if and only if each $a_i$ has the same power of $2$ in its prime decomposition.
\item Suppose $G$ is of type $B_n$ or $D_n$. If $G$ is adjoint or $G\simeq \SO(2n)$ then $g^{o(w)}=1$. 
Otherwise see Section \ref{s:more}.
\item Suppose  $G$ is of exceptional type. If $G$ is of type ${}^3 D, G_2, E_6, {}^2 E_6,  E^{ad}_7$ or $E_8$ 
then $g^{o(w)}=1$. The same holds if $G$ is of type $F_4$ and $w$ is not in the class $A_3+\tilde A_1$. 
\end{enumerate}
\end{namedtheorem}
For a more precise but more technical result see Proposition \ref{p:remaining}.

For brevity we've stated these results over an algebraically closed
field. For various weaker conditions see Proposition \ref{p:other}.

There are several key tools.  The Tits group comes with a canonical
set-theoretic splitting $\sigma:W\mapsto\tits$, and we make frequent
use of an identity in the Tits group: if $o(w)=2$ then
$\sigma(w)^2=(w\ch\rho-\ch\rho)(-1)$ (Lemma \ref{l:involutions}).   In particular if $w_0$ is the
longest element of $W$ then $\sigma(w_0)^2=z_G\in Z(G)$ and
$\oad(w_0)=o(w_0)=2$.  See Section \ref{s:involutions}.  Theorems B
and C reduce to this, by an easy calculation in some cases, or using
the theory of good elements of conjugacy classes to reduce to
principal involutions in Levi factors. See Section \ref{s:good}.

We originally computed $o(\sigma(w))$ for elliptic elements in the exceptional groups 
{\it Atlas of Lie Groups and Representations} software \cite{atlaswebsite}.
This independently confirms Theorem C (4); the two proofs rely on 
independent computer calculations.

Sean Rostami has some recent results which overlap these \cite{rostami}. 

The authors wish to thank Mark Reeder for originally asking about the
orders of lifts of Weyl group elements, and for extensive discussions
during the preparation of this article.

\sec{The Tits group}
\label{s:tits}

It is convenient to allow $F$ to be an arbitrary field, 
and suppose $G$ is a
connected, reductive algebraic group defined over $F$. Furthermore we
assume $G$ is split over $F$, and fix an $F$-split Cartan subgroup of
$G$. Then $N=\Norm_G(T)$ and $W=N/T$ are defined over $F$.  If $F$ is
algebraically closed then all Cartan subgroups of $G$ are conjugate
and $F$-split. We identify $G,T,N$ and $W$ with their $F$-points
$G(F), T(F), N(F)$, and $W(F)=N(F)/T(F)$. 

Let $X^*(T)$, or simply $X^*$, be the character lattice of $T$, $X_*=X_*(T)$ the co-character lattice, with natural perfect pairing 
$\langle \,,\,\rangle:X^*\times X_*\rightarrow\Z$. 
Write $\Delta\subset X^*$ for the roots of $T$ in $G$.
If $B$ is a Borel subgroup containing $T$ it defines a set
of positive roots $\Delta^+$ of $T$ in $G$, with associated simple roots $\Pi$. 
The Weyl group is generated by $\{s_\alpha\mid \alpha\in \Pi\}$, with the braid relations
and $s_\alpha^2=1$. If we've numbered the simple roots we write $s_i=s_{\alpha_i}$. 

For $\alpha\in \Delta$ let $\ch\alpha\in X_*$ be the corresponding
co-root, and set 
$m_\alpha=\ch\alpha(-1)$. The elementary abelian two-group generated by
$\{m_\alpha\mid \alpha\in \Pi\}$ is denoted $T_2$.  If $G$ is
simple and simply connected this is the set of elements of order $2$
in $T$.  In general it is the image of $T^{sc}_2$ where $T^{sc}$ is a Cartan
subgroup of the simply connected cover of the derived group of $G$.

Now fix a set $\{X_\alpha\mid \alpha\in \Pi\}$ of simple root vectors,
so $\P=(T,B,\{X_\alpha\})$ is a pinning.  Associated to $\P$ is the
Tits group $\tits=\tits_\P$.  This is a subgroup of $N$, generated by elements $\{\sigma_\alpha\mid \alpha\in \Pi\}$, 
where $\sigma_\alpha$ is a certain lift of $s_\alpha$ to $N$. 
See \cite{tits_group}.

\begin{lemma}
\label{l:tits}
\hfil
\begin{enumerate}
\item 
The Tits group $\tits$ is given by generators $\{\sigma_\alpha\mid\alpha\in\Pi\}$ and $T_2$, 
and relations
\begin{enumerate}
\item $\sigma_\alpha^2=m_\alpha$,
\item the braid relations,
\item $\sigma_\alpha t \sigma_\alpha\inv=s_\alpha(t)$\quad $(\alpha\in \Pi,h\in T_2)$.
\end{enumerate}
\item The map $\sigma_\alpha\rightarrow s_\alpha$ induces an exact sequence
$$
1\rightarrow T_2\rightarrow \tits\rightarrow W\rightarrow 1.
$$
\item If $w\in W$, choose a reduced expression 
$w=s_{\alpha_1}\dots s_{\alpha_n}$ and define $\sigma(w)=\sigma_{\alpha_1}\dots\sigma_{\alpha_n}$.
This is independent of the reduced expression, and $w\mapsto\sigma(w)$ is  a set-theoretic splitting of the exact sequence in (2).
\end{enumerate}

\end{lemma}

We also consider the twisted situation.  We say an automorphism of $G$
is distinguished if it fixes a pinning.
Suppose $\delta$ is a distinguished automorphism of $G$, of finite order, and 
it fixes a pinning $\P=(T,B,\{X_\alpha\})$. 
Define $\deltaG=G\rtimes\langle\delta\rangle$ (we identify the automorphism $\delta$ of $G$ 
with the element $(1,\delta)$ of $\deltaG$). 
Then $\delta$ induces automorphisms, also denoted $\delta$,
of the set of simple roots $\Pi$, the Dynkin diagram, $N$ 
and the Weyl group $W$. If $G$ is semisimple then $\delta$ is
determined by an automorphism of the Dynkin diagram. 
Define $\deltaN=\Norm_{\deltaG}(T)$ and $\deltaW=\deltaN/T$. Then $\deltaN\simeq N\rtimes\langle\delta\rangle$ and $\deltaW\simeq W\rtimes\langle\delta\rangle$. 

It is easy to see that $\delta$ induces an automorphism of $\tits$,
satisfying $\delta(\sigma_\alpha)=\sigma_{\delta(\alpha)}$
($\alpha\in\Pi$), and 
$\delta(\sigma(w))=\sigma(\delta(w))$ ($w\in W$). 
Define the extended Tits group $\deltaT=\tits\rtimes\langle\delta\rangle$, so again
there is an exact sequence $1\rightarrow T_2\rightarrow
\deltaT\rightarrow\deltaW\rightarrow 1$.
The splitting $\sigma:W\rightarrow\tits$ extends to $\deltaW$ by setting $\sigma(\delta)=\delta$.

Assume $\delta^2=1$. An element $w\delta\in\deltaW$ is an involution if and only if $w\delta(w)=1$, in which case,
as in \cite{GKP}, we say 
$w$ is a $\delta$-twisted involution.
More generally if $\delta^r=1$ then $(w\delta)^r=1$ if and only if $w\delta(w)\delta^2(w)\dots\delta^{r-1}(w)=1$.

Some of the main results apply without assuming $F$ is algebraically closed. 

\begin{proposition}
\label{p:other}
Let  $F$ be an arbitrary field.
\begin{enumerate}
\item  Suppose $G$ is an $F$-split,
  connected, reductive algebraic group, and $T$ is an $F$-split Cartan subgroup.  Then Theorems  B and C hold.
\item Suppose $F=\R$ and $G(\R)$ is compact. Equivalently, suppose $G$ is a compact connected Lie group. 
Then Theorems A,B and C hold.
\end{enumerate}  
\end{proposition}

The proofs of Theorems B and C hold
only assuming $T$ is $F$-split. 
(If $\tchar(F)=2$ then Theorem A holds. Otherwise it 
requires that $F$ contain certain roots of unity.) 
Statement (2) follows from:

\begin{lemma}
\label{l:compact}
Suppose $G$ is a connected compact Lie group, and $T$ is a
Cartan subgroup.  Let $(G(\C),T(\C))$ be the complexification of $G$
and $T$, and choose a Borel subgroup $B(\C)$ containing $T(\C)$. 
Then we can choose a pinning $\P=(T(\C),B(\C),\{X_\alpha\})$ such that 
the Tits group $\tits_\P$ is contained in $\Norm_{G(\R)}(H(\R))$. 
\end{lemma}

\begin{proof}
This is just a version of the standard result that if $G$ is compact
then the $\Norm_{G}(T)/T\simeq \Norm_{G(\C)}(T(\C))/T(\C)$.  
To be
precise: choose $\{X_\alpha\}$ so that
$[X_\alpha, \sigma(X_\alpha)]=-\ch\alpha$, where $\sigma$ is complex
conjugation of $\text{Lie}(G(\C))$ with respect to $\text{Lie}(G)$.
\end{proof}

We dispense here with a case in which it is easy to compute $\o(w,G)$. 

\begin{lemma}
\label{l:odd}
Suppose $w\in W\delta$ has odd order. Then $\o(w,G)=o(w)$.
\end{lemma}

\begin{proof}
This is an immediate consequence of the Zassenhaus-Schur Lemma applied to the cyclic group generated by any lift of $w$. 
Concretely, let $d=o(w)$, and choose any lift $g$. If  $g^d=1$ then we are done. Otherwise replace $g$ with $g^{d+1}$: $(g^{d+1})^d=(g^{2d})^{\frac{d+1}2}=1$. 
\end{proof}

We also mention a basic reduction to simple groups, 
using the following Lemma,
which is 
proved in the same way as \cite[Lemma 2.7]{GKP}.

\begin{lemma}
Suppose $G=G_1\times G_1\times\dots \times G_1$, with $r$ factors, and $\delta$ acts
cyclically on the factors, so $\delta^r$ is an
automorphism of   the first factor.
Write $W=W_1\times\dots\times W_1$ for the Weyl group.
Then the twisted Weyl groups $\deltaW$ and 
$\negthinspace\negthinspace\phantom{a}^{\delta^r}W_1$ are defined, and
there is a natural bijection
$$\{\delta-\text{twisted conjugacy classes in }W\} \longleftrightarrow 
\{\delta^r-\text{twisted conjugacy classes in }W_1\}.
$$
\end{lemma}

\sec{Involutions}
\label{s:involutions}
Suppose $G$ is as in Section \ref{s:tits}, $\delta$ is a distinguished
automorphism of finite order of $G$, and
$\deltaW=W\rtimes\langle\delta\rangle$.

\begin{lemma}
\label{l:involutions}
Suppose  $w\in W\delta$ is an involution. 
Then $\sigma(w)^2=(w\ch\rho-\ch\rho)(-1)$. 
If $w_0$ is the longest element of $W$ then $(\sigma(w_0\delta))^2=z_G\delta^2$.
\end{lemma}

\begin{proof}
For the first assertion,  by assumption $\delta(w)=w\inv$ and $\delta^2=1$, so 
$(\sigma(w)\delta)^2=\sigma(w)\delta\sigma(w)\delta
=\sigma(w)\sigma(\delta(w))\delta^2
=\sigma(w)\sigma(w\inv).
$
Apply \cite[Lemma 5.4]{contragredient}. The second statement follows from this, and the fact that 
$w_0$ and $\sigma(w_0)$ are fixed by every distinguished automorphism \cite[Lemma 5.3]{contragredient}.
\end{proof}

Let $S$ be a subset of the simple roots, with corresponding
Levi factor $L(S)$ and Weyl group $W(S)$.  Then the pinning for $G$
restricts to a pinning for $L(S)$, and the Tits group for $L$ embeds
naturally in that for $G$. If $S$ is $\delta$-stable the same holds
for the extended Tits groups.  Let $w_0(S)$ be a longest element of
the Weyl group $W(S)$. Let $\ch\rho(S)$ be one-half the sum of the
positive coroots of $L(S)$, and let  $z_{S}=z_{L(S)}=(2\ch\rho(S))(-1)$ be the
principal involution in $L(S)$.

The preceding Lemma applied to $L(S)$ gives:

\begin{lemma}
\label{l:w_0(S)}
Suppose $S\subset \Pi$ is a  set of simple roots. Then
$$
\sigma(w_0(S))^2=(2\ch\rho(S))(-1)=z_{S}.
$$  
If $\delta$ is a distinguished involution and $S$ is $\delta$-stable then 
$\delta(\sigma(w_0(S)))=\sigma(w_0(S))$ and 
$(\sigma(w_0(S))\delta)^2=z_S\delta^2$.
\end{lemma}
  
\begin{lemma}
\label{l:-1}

Suppose $\delta^2=1$ and  $w\in W\delta$ acts by inverse on $T$. Then $w$ is elliptic, and if $g$ is any lift of $w$ then $g^2=z_G$.
Furthermore 
$$
\o(w,G)=o(\sigma(w))=
\begin{cases}
2&\ch\rho\in X_*(T)\\  
4&\text{otherwise}
\end{cases}
$$
\end{lemma}
This is an immediate consequence of Lemma \ref{l:involutions}.
\sec{Lifting of the Weyl group}
\label{s:lifting}

In this section we assume $F$ is algebraically closed. 

We say {\it $W$ lifts to $G$} if the exact sequence \eqref{e:exactW} splits,
i.e. there is a group homomorphism $\phi:W\rightarrow N$ satisfying $p(\phi(w))=w$ for all $w\in W$. 
If this holds then  $W$ is isomorphic to a subgroup of $N$, and 
{\it a fortiori} $\o(w,G)=o(w)$ for all $w\in W$, and 
$o(\sigma(w))=o(w)$ for all elliptic $w\in W\delta$.

The case of characteristic $2$ is easy.

\begin{proposition}
Suppose $F$ has characteristic $2$. Then the Tits group $\tits\subset N$
is isomorphic to $W$. 
\end{proposition}

\begin{proof}
By the exact sequence in Lemma \ref{l:tits}(2) the kernel of the map from $\tits$ to $W$ 
is $T_2$. But $T_2$ is generated by 
the elements $\ch\alpha(-1)$, all of which are trivial in
characteristic $2$.
\end{proof}

For the remainder of this section we assume $\tchar(F)\ne 2$, and
determine the  simple groups $G$ for which \eqref{e:exactW} splits.

We first address the question of the uniqueness of a splitting.
Let $R\subset X^*$ be the root lattice, and $\ch R\subset X_*$ the coroot lattice.  Set $Z=Z(G)$.

\begin{lemma}
\label{l:connectedtorus}
Fix $\mu\in X_*(T)_\Q$. Define 
$$
\mu^\perp=\{\gamma\in R\mid \langle\gamma,\mu\rangle=0\}
$$
and 
$$
S=\bigcap_{\gamma\in \mu^\perp}\ker(\gamma)\subset T.
$$
Then $S/Z$ is a (connected) torus.
If $\mu$ is a coroot  then $\dim(S/Z)=1$.
\end{lemma}

\begin{proof}
It is straightforward to see 
that $X^*(S)=X^*(T)/\mu^\perp$. 
If $G$ is adjoint then $X^*(T)$ is the root lattice $R$. It is obvious that
$R/\mu^\perp$ is torsion free, which implies $S$ is connected. 
In general
$S$ is the inverse image of a connected torus in $\Gad$, so $S=S_0Z$.

Suppose $\alpha\in \Delta$. After passing to the dual root system if
necessary we may assume $\alpha$ is long, and after conjugating by $W$
that it is the highest root. Except in type $A_n$ the highest root is
orthogonal to all but $1$ simple root. In type $A_n$ $(n\ge 2)$
$\alpha$ is orthogonal to $n-2$ simple roots. If $\delta,\epsilon$
are the remaining two simple roots then $\langle
\delta-\epsilon,\ch\alpha\rangle=0$. This proves the final assertion.
\end{proof}

\begin{lemma}
\label{l:findw}
Fix  $\alpha\in\Delta$. Suppose $t\in T$
satisfies: $\beta(t)=1$ for all $\beta\in(\ch\alpha)^\perp$.   
Then there exists $w\in F^\times$ such that $\ch\alpha(w)t\in Z$.
\end{lemma}

\begin{proof}
Let $S=\cap_{\gamma\in(\ch\alpha)^\perp}\ker(\gamma)$. 
Then $\ch\alpha(F^\times)\subset S$. By the Lemma above $S/Z$ is a one
dimensional torus, so the map $F^\times\overset{\ch\alpha}\longrightarrow S\rightarrow S/Z$
is surjective.
\end{proof}

Suppose $\alpha\in\Delta$. It is well known (and easy to check) that,
except in type $A_n$, 
\begin{equation}
\label{e:twolattices}
(\ch\alpha)^\perp=\Z\langle\{\beta\in\Delta\mid
\langle\beta,\ch\alpha\rangle=0\}\rangle.
\end{equation}

We need a variant of this. We only need the simply laced case.

\begin{lemma}
\label{l:lattices}
Suppose $G$ is simply laced and no simple factor is of type $A_3$ or
$D_4$. 
Fix $\alpha\in\Delta$.
Let $\Delta(\alpha)=\{\beta\in \Delta\mid \langle
\beta,\ch\alpha\rangle=0\}$. This is a root system. 
Consider the lattice $L$ spanned by 

\begin{subequations}
\renewcommand{\theequation}{\theparentequation)(\alph{equation}}  
\label{e:lattices}  
\begin{equation}
\{\beta\in \Delta(\alpha)\mid\text{ the
    simple factor of }\beta\,\text{ in }\,\Delta(\alpha)\text { is not of
  type }A_1\}
\end{equation}
and
\begin{equation}
\{2\delta+\alpha\mid \delta\in\Delta, \langle
\delta,\ch\alpha\rangle=-1\}
\end{equation}
Then $L=(\ch\alpha)^\perp$.
\end{subequations}
\end{lemma}

\begin{proof}
The containment $L\subset(\ch\alpha)^\perp$ is obvious.

Since the statements only involves roots we may assume $G$ is  simple. It is easy to check $A_1,A_2$
directly, so  (since $A_3$ and $D_4$ are excluded) we may assume $G$ is of type $A_4$, or 
$\rank(G)\ge 5$.

After conjugating we may assume
$\alpha$ is the highest root. Assume $G$ is not of type $A_n$. Then
$\alpha$ is orthogonal to all but one simple root, and
these are the simple roots of $\Delta(\alpha)$. By \eqref{e:twolattices}
it is enough to show every simple root of $\Delta(\alpha)$ is 
in the span of (a) and (b).

In types $E_6,E_7$ and $E_8$, $\Delta(\alpha)$ is connected, the simple
factor condition in (a) is trivially satisfied, and the result is immediate. In type
$D_n$, $\Delta(\alpha)$ has type $A_1\times D_{n-2}$. If $n\ge 5$
there is only one $A_1$ factor. Taking $\delta$ to be the simple root
not orthogonal to $\alpha$, it is easy to see the roots of this factor
are in the $\Z$-span of (a) and (b).

Now suppose $G$ is of type $A_n$ with $n\ge 4$. 
In this case $\Delta(\alpha)$ is of type $A_{n-2}$, and there are two
simple roots $\delta,\epsilon$ non-orthogonal to $\alpha$.
Suppose $\gamma\in (\ch\alpha)^\perp$. Then $\gamma+c(2\delta+\alpha)$ is in the
$\Z$-span of (a) 
for some choice of integer $c$. 
\end{proof}

\begin{proposition}
\label{p:splittingunique}
Suppose either:
\begin{enumerate}
\item $G$ is simply laced,  and no simple factor is of type
  $A_3$ or $D_4$, or
\item $G$ is simply connected.
\end{enumerate}
Suppose  $\phi,\phi'$ are two splittings of   \eqref{e:exactW}.
Then there exists $t\in T$ and $\{z_w\in Z\mid w\in W\}$ such that
$\phi'(w)=z_wt\phi(w)t\inv$ for all $w\in W$.

The elements $z_w$ are determined by $\{z_\alpha=z_{s_\alpha}\mid \alpha\in
\Pi\}$, where $z_\alpha\in Z_2$ (the $2$-torsion subgroup of $Z$).  If $\alpha$ is conjugate to $\beta$ then  $z_\alpha=z_\beta$.
If (1) holds then $Z_2$ acts simply transitively on the set of splittings. 
\end{proposition}

\begin{proof}
Suppose $\phi$ is a splitting, and set $g_\alpha=\phi(s_\alpha)$ ($\alpha\in \Pi$). 
Then $\phi'(s_\alpha)=t_\alpha g_\alpha$ for some $t_\alpha\in T$.

First assume (1) holds.

Fix $\alpha\in\Pi$. We claim that $\beta(t_\alpha)=1$ for all $\beta\in (\ch\alpha)^\perp$.
By Lemma \ref{l:lattices} it is enough to show $\beta(t_\alpha)=1$ for all $\beta$ in \ref{e:lattices}(1) and (2). 

First suppose $\beta$ is in (1). Since $\beta$ is orthogonal to
$\alpha$, $\{g_\alpha, g_\beta\}=1$ ($\{,\}$ denotes the commutator).  
Then $\{t_\alpha g_\alpha, t_\beta g_\beta\}=1$ if and only if 
$t_\alpha s_\alpha(t_\beta)=t_\beta s_\beta(t_\alpha)$. 
Using the fact if $t\in T$ then
$s_\alpha(t)=t\ch\alpha(\alpha(t\inv))$, the condition is equivalent
to
$$
\ch\alpha(\alpha(t_\beta))=
\ch\beta(\beta(t_\alpha)).
$$
By assumption we can find $\gamma\in\Delta$ such that 
\begin{equation}
\label{e:gamma}
\langle\gamma,\ch\alpha\rangle=0\text{ and }\langle\gamma,\ch\beta\rangle=-1
\end{equation}

Apply $\gamma$ to both sides to conclude $\beta(t_\alpha)=1$. 
  
Now suppose $\langle\delta,\ch\alpha\rangle=-1$. Since $g_\alpha^2=1$ and $(t_\alpha g_\alpha)^2=1$ we conclude $t_\alpha s_\alpha(t_\alpha)=1$, i.e. 
$$
t_\alpha^2\ch\alpha(\alpha(t_\alpha\inv))=1.
$$
Apply $\delta$ to both sides to conclude $(2\delta+\alpha)(t_\alpha)=1$. This proves the claim.

Therefore by Lemma \ref{l:lattices} we conclude $\mu(t_\alpha)=1$ for all $\mu\in(\ch\alpha)^\perp$.
By Lemma \ref{l:findw} we can find $w_\alpha\in F^\times$ such that $\ch\alpha(w_\alpha)t_\alpha\in Z$. 
This holds for all $\alpha\in\Pi$, and we can choose $t\in T$ so that $\alpha(t)=w_\alpha$ for all $\alpha\in \Pi$. 
Set $z_\alpha =\ch\alpha(w_\alpha)t_\alpha\in Z$. 
Then 
$$
t(t_\alpha g_\alpha)t\inv =ts_\alpha(t\inv)t_\alpha g_\alpha =\ch\alpha(\alpha(t))t_\alpha g_\alpha=\ch\alpha(w_\alpha)t_\alpha g_\alpha=z_\alpha g_\alpha.
$$
Also $(t_\alpha g_\alpha)^2=g_\alpha^2=1$ implies $z_\alpha^2=1$.

Now assume (2) holds. Replace \eqref{e:lattices}(1) with the larger
set 
\begin{equation}
\label{e:lattices2}
\{\beta\in \Delta\mid \langle\beta,\ch\alpha\rangle=0\}.
\end{equation}
The 
lattice spanned by \eqref{e:lattices2} and \eqref{e:lattices}(2) is still equal
to $(\ch\alpha)^\perp$ ((2) is only needed in type $A_n$). Suppose
$\beta$ is in  \eqref{e:lattices2}. 
Since $G$ is simply connected, we can find $\gamma\in X^*(T)$
satisfying \eqref{e:gamma}, so as before we conclude
$\beta(t_\alpha)=1$. The rest of the proof is the same.

It is clear that the $z_\alpha$ have order $2$ and determine all
$z_w$. 
For the penultimate assertion, after conjugating by $t\in T$ we may assume 
$\phi'(w)=z_w\phi(w)$ for some $z_w\in W$. Suppose $\beta=w\alpha$ ($w\in W$).
Applying $\phi'$ 
to the identity $ws_\alpha w\inv =s_\beta$ gives
$$
\begin{aligned}
z_w\phi(w)z_\alpha \phi(s_\alpha)\phi(w\inv)z_w\inv =z_\beta\phi(s_\beta).
\end{aligned}
$$
Then $\phi(ws_\alpha w\inv)=s_\beta$ implies $z_\alpha=z_\beta$.
The final assertion is now clear.
\end{proof}

\begin{example}
\label{ex:PSL(4)}
Let $G=PSL(4)$. Then the conclusion of Proposition \ref{p:splittingunique} does not hold. 
Choose the diagonal Cartan subgroup, the usual simple reflections $s_i$ $(1\le i\le 3$) and 
choose a fourth root $\zeta$ of of $-1$. Then $\phi(s_i)=g_i$, where 
$$
g_1=\begin{pmatrix}
 0&\zeta&0&0\\
\zeta&0&0&0\\
0&0&\zeta&0\\
0&0&0&\zeta\\
\end{pmatrix},
\quad
g_2=\begin{pmatrix}
\zeta&0&0&0\\
0&0&\zeta&0\\
0&\zeta&0&0\\
0&0&0&\zeta\\
\end{pmatrix},
\quad
g_3=\begin{pmatrix}
 \zeta&0&0&0\\
0&\zeta&0&0\\
0&0&0&\zeta\\
0&0&\zeta&0\\
\end{pmatrix}
$$
(the image in $PSL(4)$ of these elements) 
is a splitting. Also $\phi'(s_i)=g'_i$ where 
$$
g'_1=\begin{pmatrix}
 0&1&0&0\\
1&0&0&0\\
0&0&1&0\\
0&0&0&-1\\
\end{pmatrix},
\quad
g'_2=\begin{pmatrix}
-1&0&0&0\\
0&0&1&0\\
0&1&0&0\\
0&0&0&1\\
\end{pmatrix},
\quad
g'_3=\begin{pmatrix}
1&0&0&0\\
0&-1&0&0\\
0&0&0&1\\
0&0&1&0\\
\end{pmatrix}
$$
is another splitting, not conjugate to $\phi$. 

The splitting $\phi$ is the image of the splitting by permutation matrices into $GL(4)$, composed 
with the maps $GL(4)\rightarrow GL(4)/Z(GL(4))\simeq SL(4)/Z(SL(4))=PSL(4)$.
On the other hand $\phi'$ 
is the splitting of $W$ into  $SO(6)$ discussed below, composed with $SO(6)\rightarrow SO(6)/\pm I\simeq PSL(4)$.
\end{example}

It turns out that this is the {\it only} (simple) case where the lifting is not unique up to conjugacy and multiplication by $Z_2$. See Corollary \ref{c:Zaction1}.

Before turning to the main result, we dispense with a few cases where it is
easy to prove that $W$ does not lift to $G$.

\begin{lemma}
\label{l:easynonlift}
Suppose $w\in W\delta$ is an elliptic element and $o(\sigma(w))=2o(w)$. Then $W$ does not lift to $G$.
If $\ch\rho\not\in X_*(T)$ then $W$ does not lift to $G$.
\end{lemma}
This is immediate; the last line is from Lemma \ref{l:-1}.

\begin{lemma}
\label{lemma:easy}
The Weyl group does not lift to $SL(2n), Sp(2n)$ or $\Spin(n)$. 
\end{lemma}

\begin{proof}
In types $A_{2n+1}, B_n$ and $C_n$ $\ch\rho$ is not in the root lattice, 
i.e. $X_*(T)$ for the simply connected group.

Suppose $G=\Spin(n)$.
Associated to the partition
$(2,1,\dots,1)$ of $n$ there is an elliptic element $w\in W$ (if $n$
is odd) or twisted elliptic element $w\in W\delta$ (if $n$ is even),
of order $4$ but whose lift has order $8$. Therefore $W$ does not lift
to $\Spin$.
\end{proof}

\begin{lemma}
\label{l:W(H,T)}
Suppose $H$ is a subgroup of $G$ containing $T$.
If $W(G,T)$ lifts to $G$ then the exact sequence $1\rightarrow T\rightarrow \Norm_H(T)\rightarrow \Norm_H(T)/T\rightarrow 1$ splits. 
\end{lemma}

This is also immediate;  a splitting of \eqref{e:exactW} restricts to give a splitting. 
We will use this to eliminate some exceptional cases. 

Finally we note a generalization of Lemma \ref{l:odd}.

\begin{lemma}
\label{l:oddcyclic}
Suppose  $A\subset Z$ is a cyclic group of odd order.
If $W$ lifts to $G/A$ then $W$ lifts to $G$.
\end{lemma}

\begin{proof}
Identifying $W$ with a subgroup of $G/Z$ via a splitting, 
and taking the inverse image $\wt W$ in $G$, 
we have an exact sequence
\begin{equation}
\label{e:Wtilde}
1\rightarrow A\rightarrow \wt W\rightarrow W\rightarrow 1
\end{equation}
Let $m=|A|$. 
The exact sequence of trivial $W$-modules
$$
1\rightarrow A\rightarrow F^\times \overset m\longrightarrow
F^\times\rightarrow 1
$$
gives rise to the exact sequence
$$
H^1(W,F^\times)\rightarrow H^2(W,A)\rightarrow H^2(W,F^\times)
$$
The middle term is killed by $m$. 
On the other hand $H^1(W,F^\times)\simeq \Hom(W,F^\times)\simeq
\Hom(W/[W,W],F^\times)$, and this is killed by $2$. 
Also $H^2(W,F^\times)$ is killed by $2$ by \cite{ihara_yokonuma}.
Therefore $H^2(W,A)=1$, so $W$ lifts.
\end{proof}

\begin{theorem}
\label{t:Wsplit}
Assume $G$ is  simple and $\tchar(F)\ne 2$. Then \eqref{e:exactW} splits in the following cases, 
and not otherwise:
\begin{enumerate}
\item Type $A_{n}:$  $|Z(G)|$ is odd, or $G=\SL(4)/\pm I\simeq \SO(4)$.
\item Type $B_{n}:$  $G=\SO(2n+1)$ (adjoint).
\item Type $C_n$:  $n\le 2$ and
  $G=P\SL(2)$ or $PSp(4)$ (adjoint).
\item Type $D_n$:   $G=\SO(2n)$ or
  $G=P\SO(2n)$ (adjoint); also $\Semispin(8)\simeq \SO(8)$.
\item Exceptional groups:  $G$ is of type $G_2$.
\end{enumerate}
\end{theorem}

Implicit in (4) is the assertion that  $W$ does not lift to
$\Semispin(4m)$, unless $m=2$.

When $F=\C$ this was proved in 
\cite[Theorem 2]{Weyl_splitting}, omitting a few cases in types $A_n$ and $D_n$, using 
case-by-case calculations in the braid group.
Here is a complete proof, including the missing cases, and relying as
little as possible on braid group calculations.

\begin{proof}[Proof of the Proposition]
We only consider cases which are not already handled by Lemma  \ref{l:easynonlift}.
\medskip

\noindent{$G=PSp(2n)$ (adjoint).}
If $n=1$  $G\simeq \SO(3)$,
and if $n=2$ $G\simeq \SO(5)$. In both cases $W$ lifts
(see the next case). 

Assume  $G=PSp(2n)$ (adjoint) and  $n\ge 3$.  Embed
$G_1=Sp(4)\times \SL(2)^{n-2}$ in $Sp(2n)$ in the usual way. Let $w$ be
the Coxeter element of $W(G_1)$. This has order $4$ and is elliptic.
It is easy to see that if $g$ is a lift of $w$ to
$W(G_1)$ then $g^4\ne -I$, so the image of $g$ in  $G_1/\pm I\subset
PSp(2n)$ also has order $8$. 
By Lemma \ref{l:easynonlift} $W$ does not lift to $G$. 
See Section \ref{s:more}; this is the case of the partition $(2,1,\dots,1)$ of $n$.

\medskip
\noindent{$G=\SO(n)$ and $PSO(n)$}.
Let $G=\SO(V)$ where
$V$
is a non-degenerate orthogonal space of dimension $n$.
Write $V=X\oplus V_0\oplus Y$ where
$X,Y$
are maximal isotropic subspaces, in duality via the form, and $V_0$
is isotropic of dimension $r\in \{0,1\}$.
Let $\{e_1,\dots, e_m\}$ be a basis of $X$, with dual basis $\{f_1,\dots, f_m\}$ of $Y$.
Let $S=\{e_1,\dots, e_m,f_1,\dots, f_m\}$. 
If $V_0\ne 0$ choose a nonzero vector $e_0\in V_0$. 
Then the subgroup $T\subset G$ stabilizing $V_0$ and each line $F\langle e_i,f_i\rangle$ is a Cartan subgroup of $G$.
Furthermore the subgroup $\{g\in G\mid g(S)=S, ge_0=e_0\}$ normalizes $T$, and is a lifting of $W$ to $G$.

Therefore {\it a fortiori} $W$ lifts to the adjoint group.

\medskip
\noindent{$G=\Semispin(4n)$}

The center of $\Spin(4n)$ is the Klein four-group.  Let $\tau$ be an
automorphism of order $2$ of $\Spin(4n)$ coming from an automorphism of
the Dynkin diagram (which is unique unless $n=2$). 
Write $Z(\Spin(4n))=\{1,x,y,z\}$ where
$\tau(x)=y$ and $\tau(z)=z$. Then $\Spin(4n)/\langle z\rangle\simeq \SO(4n)$. 
On the other hand $\Spin(4n)/\langle x\rangle 
\simeq \Spin(4n)/\langle y\rangle$, and this group is  denoted
{\it Semispin(4n)}.

\begin{example}
\label{ex:spin4}
Take $n=1$, so $G=\Spin(4)\simeq \SL(2)\times \SL(2)$, with $\tau$
exchanging the factors; set $x=(I,-I), y=(-I,I)$ and $z=(-I,-I)$. 
Since $W$ does not lift to $\SL(2)$ it obviously does not lift to
$\Spin(4)$, or $\Semispin(4)\simeq P\SL(2)\times \SL(2)$.

If $s,t$ are the simple reflections in the first and second factors,
take $g_s=(
\begin{pmatrix}
0&1\\-1&0
\end{pmatrix},\diag(i,-i))
$ and 
set $g_t=\tau(g_s)$. 
Then $g_s^2=g_t^2=(-I,-I)=z$, so $W$ lifts to $\Spin(4)/\langle
z\rangle\simeq \SO(4)$. 
\end{example}

If $n=2$, so $G$ is  of type $D_4$,
the three elements $x,y,z$ of $\Spin(8)$ are related by
automorphisms of $\Spin(8)$. Since $\Spin(8)/\langle z\rangle\simeq
\SO(8)$, we conclude
$\Semispin(8)\simeq \SO(8)$, and $W$ lifts by the previous discussion. 
So assume $n\ge 3$.
Fix a simple root $\alpha$.  Let $g_\alpha\in \SO(4n)$ be the image of
$s_\alpha$ discussed above.  Assume $W$ lifts to $\Semispin(4n)$, and
let $h_\alpha\in \Semispin(4n)$ be the image of $s_\alpha$.

We proceed by contradiction, using Proposition \ref{p:splittingunique}
and the following diagram, to reduce to the case $n=1$.

$$
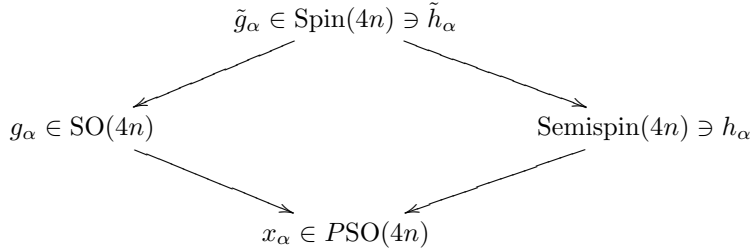
\xymatrix{
&\tilde g_\alpha \in \Spin(4n)\ni \tilde h_\alpha\ar[dl]\ar[dr]\\
g_\alpha\in \SO(4n)\ar[dr]&&\Semispin(4n)\ni h_\alpha\ar[dl]\\
&x_\alpha\in P\SO(4n)
}
$$

By Proposition \ref{p:splittingunique} the
images of $g_\alpha$ and $h_\alpha$ in $\PSO(4n)$ are $T$-conjugate, so 
after conjugating the splitting into $\Semispin(4n)$ we may assume these are equal. 

Now let $\tilde g_\alpha, \tilde h_\alpha$ be inverse images of
$g_\alpha,h_\alpha$ in $\Spin(4n)$.  By the preceding discussion
these have the same image in 
 $\PSO(4n)$,
so they differ by an element of the center. Since the center is a
two-group, $\tilde g_\alpha^2=\tilde h_\alpha^2$.

Obviously $\wt g_\alpha^2\in \{1,z\}$ where $z\in Z(\Spin(4n))$ is the
nontrivial element of the kernel of the map to $\SO(4n)$.  It is
enough to show $\wt g_\alpha^2\ne 1$, for then $\wt h_\alpha^2=z$, so
its image $h_\alpha$ in $\Semispin(4n)$ is nontrivial.
This follows by a reduction to $\Spin(4)$.

Take a subgroup 
$H\simeq \SO(4)\times SO(4m-4)\subset \SO(4m)$, where the $\alpha$-root space is 
contained in the $SO(4)$ factor. Then, by our choice of splitting of $W$ in $\SO(4m)$ discussed above,  
$g_\alpha=(u,1)\in \SO(4)\times \SO(4m-4)$.  Let $(v,w)$ be an inverse
image of $(u,1)$ in $\Spin(4m)\times \Spin(4m-4)$.  Then
$w\in Z(\Spin(4m-4))$, and by Example \ref{ex:spin4} (and Proposition
\ref{p:splittingunique} again) $v^2$ is a non-trivial element  of
the center of $\Spin(4)$.  Therefore $(v,w)^2=(u^2,1)\ne (1,1)$.  The inverse
image of $H$ in $\Spin(4m)$ is isomorphic to
$\Spin(4)\times \Spin(4m-4)/\langle z_1,z_2\rangle$ where $z_1,z_2$
are non-trivial. It follows that $\tilde g_\alpha^2$, i.e. the image of $(v^2,1)$ in $\Spin(4m)$, is non-trivial.

\medskip
\noindent{$\SL(n)$}:

For $1\le i\le n-1$ let  $p_i\in GL(n)$ be the permutation matrix 
$$
p_i(x_1,\dots, x_i,x_{i+1},\dots, x_n)=
(x_1,\dots, x_{i+1},x_{i},\dots, x_n).
$$
Write $s_i$ for the corresponding simple reflections in $W$.
The map $\phi_{\GL}(s_i)=p_i$, extends to a splitting $W\rightarrow \GL(n)$.

If $n$ is odd then $\phi(s_i)=-p_i$ is a splitting into $SL(n)$, so
assume $n$ is even. We already know $W$ does not lift to $SL(n)$. 
If $n=2$ then $\Gad\simeq SO(3)$, and if $n=4$ $SL(4)/\pm I\simeq SO(4)$, 
so $W$ lifts in these cases, and $PSL(4)$.

So assume $n\ge 6$, and suppose $A\subset Z(SL(n))$.
We identify $A$ with a subgroup of $\mu_n(F)$. 
Then $\phi_{\GL}$ factors to a splitting $W\rightarrow \GL(n)/A$.
Suppose there exists $z\in F^\times$ such that $\det(zp_i)=1$ and
$(zp_i)^2\in A$. Then $\phi(s_i)=zp_iA$ is a splitting $W\rightarrow
\SL(n)/A$. By Proposition \ref{p:splittingunique}  
this condition is both necessary and sufficient for the existence of a
splitting.

The condition holds if and only if there exists $z\in F^\times$ satisfying
\begin{equation}
z^2\in A\text{ and } z^n=-1.
\end{equation}

Then $(z^2)^n=1$ so the order of $z^2$ divides $n$. 
Write $n=n_2q$ with $n_2=2^k$ (since $n$ is odd $k\ge 1$) and $q$
odd. Thus $(z^2)^{n_2q}=1$, but $(z^2)^{\frac{n_2}2q}=-1$.
This implies $n_2$ divides the order of $z^2$, so $n_2$ divides the
order of $A$. Therefore $|Z/A|$ is odd.

\medskip
\noindent{$G_2$}: 
Label the simple roots $\alpha_1,\alpha_2$. 
For $i=1,2$ the subgroup generated 
generated by $T$ and the root groups for $\pm\alpha_i$ is isomorphic
to $\GL(2)$, so $s_i$ has a lift to an involution $n_i$. 
The long element of the Weyl group is $w_0=(s_1s_2)^3$. 
By Lemma \ref{l:involutions} $(n_1n_2)^6=(2\ch\rho)(-1)=1$. It follows 
that $n_1,n_2$ generate a lift of $W$ in $G$.
\medskip

For the remaining exceptional groups we
choose a subgroup $H$ to be the centralizer of an
element of $T$ of order $2$, so that $W(H,T)$ does not
lift to $H$. 
These groups are well understood, for example see
\cite[Chapter 5, \S1]{ov}.
Then 
Lemma \ref{l:W(H,T)} implies $W(G,T)$ does not lift to $G$.

\smallskip
\noindent $F_4$: It is well known that $F_4$ contains a subgroup $H\simeq \Spin(9)$, and we already know 
$W$ doesn't lift to $\Spin(9)$. 

\smallskip
\noindent $E_6:$ 
The center of the simply connected group is cyclic of order $3$,   by Lemma \ref{l:oddcyclic} 
we may assume $G$ is simply connected. 
Let $H$ be the subgroup of type $A_1\times A_5$. 
Then 
$H\simeq \SL(2)\times \SL(6)/\langle(-I,-I)\rangle$.
Suppose the simple reflection in the first factor lifts to an element of $H$, with
representative $(g,h)\in \SL(2)\times \SL(6)$. Then $g^2=-I$ so if the image of $(g,h)^2$ is trivial in $H$  then
$h^2=-I$. But clearly $h\in Z(\SL(6))$ 
and there is no element  in   $Z(\SL(6))$ with this property.

Since the center of the simply connected group is $\Z/3\Z$, if $W(\Gad,T)$ lifts to $\Gad$ 
then it lifts to the simply connected group by Lemma \ref{l:oddcyclic}.

\smallskip
\noindent $E_7:$ Take $H$ of type $A_7$. Then $H\simeq SL(8)/A$ where
has order $2$ or $4$, depending on whether $E_7$ is simply connected
or adjoint, so $|Z(H)|$ is $2$ or $4$, and by (1) of the Proposition $W(H,T)$
does not lift to $H$.

\smallskip
\noindent $E_8:$ Take $H$ of type $D_8$. It is well known that $H\simeq \Semispin(16)$,
so $W(H,T)$ does not lift to $H$.

This concludes the proof of  Theorem \ref{t:Wsplit}.
\end{proof}

\begin{corollary}
\label{c:Zaction1}
Suppose $G$ is  simple, and $W$ lifts to $G$.
\begin{enumerate}

\item If $G=PSO(6)\simeq PSL(4)$ there are two $T$-conjugacy classes of splittings. 
\item If $G=SO(2n)$ there are two $T$-conjugacy classes of splittings, related by multiplication by $-I\in Z$.
\item In $G_2$ and all other simply laced cases there is one $T$-conjugacy class of splittings.
\end{enumerate}
\end{corollary}

\begin{proof}

Most cases follow from a combination of the Theorem and 
Proposition \ref{p:splittingunique}.

Suppose $G$ is of type $A_n$ with $n\ne 3$. By Proposition \ref{p:splittingunique} 
the lift is unique up to conjugacy by $T$ and multiplication by $Z_2$. However
by 
Theorem \ref{t:Wsplit} the assumption that $W$ lifts to $G$ implies $Z_2$ is trivial. 

If $G=PSL(4)$ then there are two non-conjugate splittings given in Example
\ref{ex:PSL(4)}. It is straightforward to see these are the only ones up to $T$-conjugacy, 
and the lifting to $SO(4)$ is unique up to $T$-conjugacy and the center.

In $SO(2n)$ $|Z_2|=2$ and in $G_2$ the center is trivial. The only
other exceptional case is $D_4$. It follows from a tedious and not
very enlightening argument that $W$ lifts to $SO(8)$, uniquely up to
$T$-conjugacy and multiplication by $-I$, and the lifting to $PSO(8)$
is unique up to $T$-conjugacy. We leave the details to the reader.

\end{proof}

\sec{Coxeter elements and elliptic conjugacy classes}
\label{s:coxeterandelliptic}

\subsec{Coxeter and twisted Coxeter elements}
\label{s:coxeter}

Choose an ordering $1,\dots, n$ of the simple roots.
The corresponding Coxeter element is $\Cox=s_1s_2\dots s_n$. 
All Coxeter elements are conjugate, regular and elliptic.

Now suppose $\delta$ is a distinguished automorphism. 
Write $i_1,\dots, i_k$ for representatives of the $\delta$-orbits on
the simple roots. A twisted Coxeter element is defined to be $\Cox'=s_{i_1}\dots s_{i_k}\delta\in W\delta$. 
These elements are all $W$-conjugate, elliptic and regular.
See \cite[Theorem 7.6]{springer}. 

\begin{proposition}
\label{p:cox}
Suppose $g\in G$ is a lift of $\Cox$. Then $g^{o(\Cox)}=z_G$.
Suppose $g\in G\delta$ is a lift of $\Cox'$. Then
$g^{o(\Cox')}=z_G$. 
\end{proposition}

Since the (twisted) Coxeter elements are regular this follows from Proposition \ref{p:regular}.

It is convenient to formulate a variant of this in type $A$, using the
fact that we can take $-1$ for the outer automorphism of the root
system. 

Set $G=\SL(n,\C)$, with the usual diagonal Cartan subgroup and Borel
subgroup, and Weyl group $W$. Set 
\begin{equation}
\label{e:x}
x=
\begin{pmatrix}
0&0&\dots&0&1\\  
0&0&\dots&-1&0\\  
\vdots&\vdots&\dots&\vdots&\vdots\\
(-1)^{n+1}&0&\dots&0&0
\end{pmatrix}
\end{equation}
Then $\deltaG=\langle G,\delta\rangle$ where $\delta g\delta^{-1}=x(\gti) x\inv$ and $\delta^2=1$,
and similarly $\deltaW=\langle W,\delta\rangle$.

Let $\epsilon=x\delta$. Then 
\begin{equation}
\label{e:epsilon}
\deltaG=\langle G,\epsilon\rangle,\quad \epsilon g\epsilon^{-1}=\gti, \epsilon^2=z_G.
\end{equation}

\begin{lemma}
Suppose $G$ is of type $A_{n-1}$, and let $\Cox$ be a Coxeter element of $W$.
If $n$ is odd then $\Cox\cdot\epsilon$ is an elliptic regular element
of $W\delta$, $o(\Cox\cdot\epsilon)=2n$, and 
if $g$ is any lift of $\Cox\cdot\epsilon$ then $g^{2n}=z_G$.
\end{lemma}

\subsec{Elliptic conjugacy classes in the classical groups}
\label{s:ellclassical}

We use these results to 
describe the elliptic conjugacy classes in the classical Weyl groups. 
See \cite[Section 3.4]{geck_pfeiffer} or \cite[Section 3]{GM} for the untwisted cases, and \cite[Section 3 \& 4]{GKP} or \cite[Section 7]{He}
for the twisted ones.

\medskip
\noindent
{\bf Type $\mathbf {A_{n-1}}$}

The only elliptic conjugacy class of $W$ is that of  the Coxeter elements.

For $m\ge 2$ let $\Cox_m$ be a Coxeter element of type $A_{m-1}$, and set $\Cox_1=1$.
Suppose  $\P=(a_1,\dots, a_l)$ is  a partition of $n$ with all odd
parts. Using \eqref{e:epsilon} set 

\begin{equation}
\label{e:EP}
\E(\P)=(\Cox_{a_1}\times\cdots\times \Cox_{a_l})\epsilon\in W\delta
\end{equation}
where $\Cox_{a_1}\times\cdots\times \Cox_{a_l}$ is embedded diagonally as usual.
It is immediate that $\E(\P)$ is elliptic, and 
$$
o(\E(\P))=2\cdot\lcm(a_1,\dots, a_l).
$$
Furthermore the map  $\P\rightarrow\E(\P)$ gives a bijection 
between partitions of $n$ with all odd parts and elliptic conjugacy
classes of $W\delta$. 

\medskip
\noindent
{\bf Type $\mathbf{B_n/C_n}$}
Let $\Cox_n$ be a Coxeter element  of $W(B_n)$.
Suppose $\P=(a_1,\dots, a_k)$ is a partition of $n$,  embed
$B_{a_1}\times\dots\times B_{a_k}$ in $B_n$ as usual, and 
set $\E(\P)=\Cox_{a_1}\times \cdots\times \Cox_{a_k}$ of $W$. Then $\E(\P)$ is
elliptic, 
and the map $\P\rightarrow\E(\P)$ defines a bijection between partitions of $n$ and conjugacy
classes of elliptic elements of $W(B_n)$.

Exactly the same result holds with type $C$ in place of type $B$.

\medskip
\noindent
{\bf Type $\mathbf{D_n}$:}
Let  $\delta_n$ be a 
distinguished automorphism of order $2$, and choose the numbering of the simple roots so that root $n$ is not fixed by $\delta_n$. 
Set
$$
\Cox'_n=s_1s_2\dots s_{n-1}\delta_n.
$$
This is the twisted Coxeter element of $W(D_n)\delta_n$, and is an
elliptic regular element of $W\delta$.

Suppose $\P=(a_1,\dots, a_k)$ is a partition of $n$, and embed
$D_{a_1}\times \dots\times D_{a_k}$ in $D_n$ as usual. Then 
$W(D_{a_1})\delta_{a_1}\times\dots \times
W(D_{a_k})\delta_{a_k}$ embeds naturally in $W(D_n)\delta_n$.
Set 
$$
\E(\P)=\Cox'_{a_1}\times\dots\times\Cox'_{a_k}.
$$
Then $\E(\P)$ is an elliptic element 
of $W(D_n)$ if $k$ is even, or $W(D_n)\delta_n$ if $n$ is odd, 
and $\P\rightarrow\E(\P)$ is a bijection between the partitions of $n$ and the 
union of the elliptic conjugacy classes of $W(D_n)$ and
$W(D_n)\delta_n$. 

\sec{Good representatives of conjugacy classes in Weyl groups}
\label{s:good}

Let $B^+$ be the braid monoid associated to the Coxeter system $(W,\Pi)$. 
Let $j: W \to B^+$ be the canonical injection identifying the
generators of $W$ with the corresponding generators of $B^+$ and
$j(w w')=j(w) j(w')$ for $w, w' \in W$ with
$\ell(w w')=\ell(w)+\ell(w')$.

The distinguished automorphism $\d$ of $G$ (and hence of $W$) induces
an automorphism  of $B^+$, which we still denote by $\d$. Define the
extended Braid monoid ${}^\d B^+=B^+ \rtimes \langle\d\rangle$. The
injection $j$ extends in a canonical way to an injection
${}^\d W \to {}^\d B^+$, which we still denote by $j$.

Following \cite{GM}, we call $w \in W\delta$ a good element if
there exists a strictly decreasing sequence
$S_1 \supsetneq \cdots \supsetneq S_l$ of subsets of $\Pi$
and even positive integers $d_1,\cdots,d_l$ such
that 
\begin{equation}
\label{e:j(w)}
j(w)^{o(w)}=j(w_0(S_1))^{d_1}\cdots j(w_0(S_l))^{d_l},
\end{equation}
where 
$w_0(S_i)$ is the longest
element of the parabolic subgroup $W(S_i)$ of $W$.

\begin{proposition}
\label{p:good}
Every conjugacy class of $W\delta$ contains a good element.
\end{proposition}

This is is proved in \cite{GM}, \cite{GKP} and \cite{He} via
case-by-case analyses, and a general proof is in \cite{he_nie}.
In fact, we may choose a good element having minimal length in the conjugacy class.

If $w$ is written as in \eqref{e:j(w)} then the image of $j(w)$ in the Tits group is $\sigma(w)$, so 
by Lemma \ref{l:w_0(S)}:
\begin{equation}
\label{e:sigma(w)}
\sigma(w)^{o(w)}=(\sum_{i=1}^l d_i \rho^\vee(S_i))(-1)=\prod_{i=1}^\ell z_{L(S_i)}{}^{d_i/2}.
\end{equation}
where $z_{L(S_i)}$ is the principal involution in the Levi factor $L(S_i)$. 

Assuming we know the $d_i$ and $S_i$ explicitly, this gives a formula
for $\sigma(w)^{o(w)}$, and (at least in the elliptic case)
$o(\sigma(w))$.  Thus we need the explicit formulas of
\cite{GM}, \cite{GKP} and \cite{He}. See Section \ref{s:more}.

\sec{Regular Elements}
\label{s:regular}

Fix a distinguished automorphism $\delta$ of $G$.
Let $\overline\Q$ be an algebraic closure of $\Q$, and set $V=X_*\otimes\overline\Q$,  and
\begin{equation}
\Vreg=\{v\in V\mid \langle \alpha,v\rangle\ne 0\text{ for all }\alpha\in \Delta\}.
\end{equation}
We say that $w \in W \delta$ is
{\it regular} if it has an eigenvector $v\in\Vreg$. In this case
if the eigenvalue of $v$ is $\zeta$, we say $w$ is {\it d-regular} if $\zeta$ has order $d$.

It obvious that both $d$ and $o(\delta)$ divide $o(w)$. The case
of $d=o(w)$ is of particular significance.

\begin{lemma}
Suppose $w\in W\delta$ is $d$-regular. 
Then $o(w)=\lcm(o(\delta),d)$.
The following conditions are equivalent:
\begin{enumerate}
\item $d=o(w)$,
\item $o(\delta)$ divides $d$,
\item $\langle w\rangle$ acts freely on the roots.
\end{enumerate}
If $w\in W$ is $d$-regular then $d=o(w)$.
\end{lemma}
The elements $w$ satisfying the conditions of the Lemma are called $\mathbb Z$-regular in \cite{rgly}.

\begin{proof}
The first assertion is proved in \cite{broue_michel} and \cite{rgly}, which gives the equivalence of (1) and (2). 
The implication (1) implies (3) 
is proved in \cite{springer_regular}, following an argument of Kostant for the Coxeter element, and (3)$\Rightarrow$(2) 
is proved in \cite{rgly}. The final assertion is the case $o(\delta)=1$.
\end{proof}

The obvious case in which $d<o(w)$ is if $d=1$, which is easy to handle.

\begin{lemma}\label{d=1} We have $d=1$ if and only if $w$ is conjugate to $\d$. 
\end{lemma}
\begin{proof} If $d=1$ then $w\ch\gamma=\ch\gamma$ for a regular element $\ch\gamma$. After
conjugating by an element of $W$, we may assume that $\ch\gamma$ is in the
dominant chamber, which implies   $w=\delta$. Conversely, if $w=x \d x \inv$,
then $w$ fixes the regular element $x\rho^\vee$, hence $w$ is
$1$-regular.
\end{proof}

We have the following result on the $d$-regular elements.

\begin{proposition}
\label{p:regular}
Let $C$ be a conjugacy class of $d$-regular elements in $W \delta$ with $d>1$. Then $C$ contains an  element $w$ so that in the Braid group
$$
j(w)^{o(w)}=j(w_0)^{2 o(w)/d} 
$$
and in the Tits group
$$
\sigma(w)^{o(w)}=z_G^{o(w)/d}\in Z(G).
$$
\end{proposition}

\begin{proof}
According to 
\cite{broue_michel}, Proposition 3.11 and 6.3 (for the untwisted and twisted cases, respectively),
$$
j(w)^d=j(w_0)^2\delta^d.
$$
Raise both sides to the power $o(w)/d$, and use the fact that
$\delta^{o(w)}=1$ to conclude the first statement, and the second is an immediate consequence of this.
\end{proof}

\begin{corollary}
If $w$ is $\mathbb Z$-regular then have $\sigma(w)^{o(w)}=z_G$, and $o(\sigma(w))=o(w)$ if and only if $\ch\rho\in X_*(T)$.  
\end{corollary}

\begin{remark}
An example in which $1<d<o(w)$  is given in \cite{broue_michel},
Proof of Proposition 6.5.  Consider ${}^2 A_5$, so $\delta$ is the
nontrivial diagram automorphism of order $2$. Let $C$ be the conjugacy class of 
$w=(s_1 s_3 s_5 s_2 s_4)^2 \delta$. It is easy to check that $o(w)=6$,
$w$ is $3$-regular, and also that $w$ is good, so by the Proposition $j(w)^6=j(w_0)^4$ and $\sigma(w)^6=z_G^2=1$.
\end{remark}

Finally, we have

\begin{proposition}
Let $w$ be a regular element. Then $\oad(w)=o(w)$.
\end{proposition}

\begin{proof}
Suppose $w$ is a $d$-regular element. 
If $d=1$ then by Proposition \ref{d=1}, $w$
is conjugate to $\delta$. By definition of $\deltaT$, the lift of $\delta$ to $\deltaT$ 
has the same order of $\delta$.

Assume $d>1$.
By
Proposition \ref{p:regular}, $w$ is conjugate to an element $w'$ with
$\sigma(w')^{o(w)}=1 \in \Gad$. We take the lifting of $w$ to be a
conjugate of $\sigma(w')$. Then the order of that lifting equals
$o(w)$.

\end{proof}

\sec{Theorem \ref{t:A}: Exceptional Cases}

We still need to prove Theorem \ref{t:A}(2) for the exceptional
groups: if $G$ is simple, adjoint, and exceptional and $w$ is
elliptic, then $\o(w)=o(w)$, except for the conjugacy class
$A_3+\wt A_1$ in $F_4$.
We include the case $\threeD$ here.  We prove a bit more: we calculate $\o(w)$ 
in the non-adjoint simple exceptional groups, i.e. simply connected of type $E_6$, $E_7$ and $\threeD$.
We have already treated $G_2$ (Theorem \ref{t:Wsplit}). 

We use the explicit lists of elliptic conjugacy classes, and formulas
for $j(w)^d$, from \cite[Section 3]{GM} (untwisted) and \cite[Section
6]{GKP} (twisted) and apply \eqref{e:j(w)} and \eqref{e:sigma(w)}. This
is a straightforward case-by-case analysis.

Recall (Lemma \ref{l:w_0(S)}) $j(w_S)^d$ contributes the term 
$$
(d\ch\rho(S))(-1)=z_L^{\frac d2},
$$
where $L=L(S)$. This  is trivial if and only if $\frac d2\ch\rho(S)\in X_*(T)$.
In particular 
if $\ch\rho\in\ch R$  we can ignore any term
$j(w_I)^d$ ($d$ even). 
This holds for any adjoint group ($F_4,E_6^{ad}, E_7^{ad}, E_8$) 
and also in $E_6^{sc}$ and (for any isogeny) $\threeD$. 

The same holds for any terms $j(w_S)^d$ provided $4|d$.
Here is the example of $F_4$.
We use notation of \cite[3.5]{GM}. The simple roots are $I=\{1,2,3,4\}$ ($3,4$ are short). 
There are $9$ elliptic conjugacy classes.

\[
\begin{tabular}{|c|c|c|c|}
\hline
Elliptic class &order & Good representative & $j(w)^{o(w)}$  
\\ \hline
$4 A_1$ & 2& $w_I$ & $j(w_I)^2$ 
\\ \hline
$D_4$ & 8& $2323432134$ & $j(w_I)^2 j(w_{3 4})^4$ 
\\ \hline
$D_4(a_1)$ & 4& $324321324321$ & $j(w_I)^2$
\\ \hline
$C_3+A_1$ &  8&$1214321323$ & $j(w_I)^2 j(w_{12})^4$
\\ \hline
$A_2+\tilde A_2$ & 3&$3214321323432132$ & $j(w_I)^2$
\\ \hline
$F_4(a_1)$ & 6& $32432132$ & $j(w_I)^2$
\\ \hline
$F_4$ & 12& $4321$ & $j(w_I)^2$
\\ \hline
$A_3+\tilde A_1$ &  4&$23234321324321$ & $j(w_I)^2 j(w_{23})^2$
\\ \hline
$B_4$ &  8&$243213$ & $j(w_I)^2$
\\ \hline
\end{tabular}
\]

By the preceding discussion  all terms are trivial except possibly
in the case $A_3+\wt A_1$, the term $\ch\rho(\{2,3\})(-1)$ coming from 
$j(w_{23})^2$.
It is easy to see
$\ch\rho(\{2,3\})=\frac 32\ch\alpha_2+2\ch\alpha_3$,
so  $o(w)=4$ and  any lift of $w$ has order $8$.
Alternatively the derived group of  $L(\{2,3\})$ is isomorphic to $Sp(4)$, 
and $\sigma(w_{23})^2=z_{Sp(4)}$, which is nontrivial. 

The preceding discussion show that in  all cases in types $E_6$ (untwisted) and  $\threeD$, $\o(w,G)=o(w)$. 
Here is a list of the remaining elliptic conjugacy classes, for which it is not obvious whether
$\o(w)=o(w)$ or $2o(w)$. 

\[
\begin{tabular}{|c|c|c|c|c|}
\hline
$G$&Elliptic class &order & $j(w)^{o(w)}$  \\\hline
$\twoE$ & $4254234565423456$ & $6$ & $j(w_I)^2j(w_{2345})^2$\\
\hline
\hline
$E_7$ & $E_7$ & $18$ & $j(w_I)^2$\\
\hline
$E_7$ & $E_7(a_1)$ & $14$ & $j(w_I)^2$\\
\hline
$E_7$ & $E_7(a_2)$ & $12$ & $j(w_I)^6j(w_{257})^2$\\
\hline
$E_7$ & $E_7(a_3)$ & $30$ & $j(w_I)^6j(w_{24})^4$\\
\hline
$E_7$ & $D_6+A_1$ & $10$ & $j(w_I)^2j(w_{24})^8$\\
\hline
$E_7$ & $A_7$ & $8$ & $j(w_I)^2j(w_{257})^2j(w_2)^4$\\
\hline
$E_7$ & $E_7(a_4)$ & $6$ & $j(w_I)^2$\\
\hline
$E_7$ & $D_6(a_2)+A_1$ & $6$ & $j(w_I)^2j(w_{13})^4$\\
\hline
$E_7$ & $A_5+A_2$ & $6$ & $j(w_I)^2j(w_{2345})^2$\\
\hline
$E_7$ & $D_4+3A_1$ & $6$ & $j(w_I)^2j(w_{24567})^4$\\

\hline
$E_7$ & $2A_3+A_1$ & $4$ & $j(w_I)^2j(w_{257})^2$\\
\hline
$E_7$ & $7A_1$ &  $2$ & $j(w_I)^2$ \\
\hline\hline 
$E_8$ & $E_8(a_7)$ & $12$ & $j(w_I)^2j(w_{2345})^2$\\
\hline
$E_8$ & $E_7(a_2)+A_1$ & $12$ & $j(w_I)^2j(w_{2345})^2j(w_{24})^8$\\
\hline

$E_8$ & $E_6(a_2)+A_2$ & $12$ & $j(w_I)^2j(w_{2345})^6$\\
\hline

$E_8$ & $A_7+A_1$ & $8$ & $j(w_I)^2j(w_{2345})^2j(w_{25})^4$\\
\hline

$E_8$ & $E_6(a_2)+A_2$ & $6$ & $j(w_I)^2j(w_{2345})^2$\\
\hline

$E_8$ & $A_5+A_2+A_1$ & $6$ & $j(w_I)^2j(w_{234578})^2j(w_{78})^2$\\
\hline

$E_8$ & $D_5(a_1)+A_3$ & $12$ & $j(w_I)^4j(w_{123456})^2$\\
\hline

$E_8$ & $2A_3+2A_1$ & $4$ & $j(w_I)^2j(w_{2345})^2$\\
\hline
\end{tabular}
\]
\smallskip

Consider types $\twoE$ and $E_8$. In both cases $\ch\rho\in\ch R\subset X_*$, so we can ignore all terms $j(w_I)^d$. 
The remaining terms are: $\twoE$: 
$S=\{2,3,4,5\}$, and 
$E_8$: $S=\{2,3,4,5\}$, $\{2,3,4,5,7,8\}$,$\{7,8\}$ or $\{1,2,3,4,5,6\}$.
The Levi factors $L$ are of type $D_4, D_4, D_4\times A_2, A_2$ or $E_6$, respectively.
In each case $z_L=1$, so $\o(w,G)=o(w)$ in these cases.

Consider type $E_7$.
After reducing each $d$ modulo $4$ we have to  determine if 
$\ch\rho\in X_*$ or $\ch\rho+\ch\rho(S)\in X_*$ where $S=\{2,5,7\}$ or $\{2,3,4,5\}$.
Using notation of \cite{bourbaki} we have:
$$
\begin{aligned}
\ch\rho&=
17\ch\alpha_1+
\frac{49}2\ch\alpha_2+
33\ch\alpha_3+
48\ch\alpha_4+
\frac{75}2\ch\alpha_5+
26\ch\alpha_6+
\frac{27}2\ch\alpha_7\\
\ch\rho(\{2,5,7\})&=\frac12\ch\alpha_2+\frac12\ch\alpha_5+\frac12\ch\alpha_7\\
\ch\rho(\{2,3,4,5\})&=3\ch\alpha_2+3\ch\alpha_3+4\ch\alpha_4+3\ch\alpha_5
\end{aligned}
$$
Since $\ch\rho(\{2,3,4,5\})\in\ch R$ we can ignore these terms. 
Also $\ch\rho$ is in the coweight lattice, and $\ch\rho\equiv\ch\rho(\{2,5,7\})\pmod{\ch R}$.
Therefore in 
 type $E_7^{ad}$, $\ch\rho$ and $\ch\rho+\ch\rho(\{2.5,7\})$ are contained in $X_*$ (the coweight lattice), so 
$\oad(w)=\o(w,G)=o(w)$ in all cases. 

Finally in type  $E_7^{sc}$ we have $X_*=\ch R$, and we see $\ch\rho\not\in\ch R$, $\ch\rho+\ch\rho(\{2,5,7\})\in\ch R$. 
We conclude that 
$\o(w,G)=2o(w)$ \emph{except} for the classes
$E_7(a_2)$, $A_7$ and $2 A_3+A_1$. 
Here is the conclusion.

\begin{proposition}
\label{p:remaining}
Suppose $G$ is simple, exceptional, and $w\in W\delta$ is an elliptic element. 

\begin{itemize}
\item If $G$ is type $G_2$ then $W$ lifts, so $\o(w)=o(w)$.
\item In types  $\threeD, E_6, \twoE$ and $E_8$ every term $z_L$ occurring is trivial, so $\o(w)=o(w)$.
\item In type $F_4$ every term $z_L$ occurring is trivial except for the conjugacy class $A_3+\wt A_1$, 
and 
$$
\o(w)=
\begin{cases}
  2o(w)&\text{conjugacy class }A_3+\wt A_1\\
  o(w)&\text{otherwise}
\end{cases}
$$
\item In type $E_7$
$$
\sigma(w)^{o(w)}=
\begin{cases}
  1& \text{conjugacy classes } E_7(a_2), A_7 \text{ and } 2A_3+A_1\\
  z_G&\text{otherwise}
\end{cases}
$$
In particular if $G$ is adjoint then 
$\o(w)=o(w)$. 
If $G$ is simply connected then
 (since $z_G\ne 1$ in $G$) ,
$\o(w)=o(w)$ only in the three conjugacy classes $E_7(a_2), A_7$ and $2A_3+A_1$, and $\o(w)=2o(w)$ otherwise.
\end{itemize}
\end{proposition}
This completes the proof of Theorem \ref{t:A}(2). 

\sec{Proof of Theorem \ref{t:B}}
\label{s:more}

For the classical groups we use the description of the (twisted) elliptic conjugacy classes (Section \ref{s:ellclassical}).

\medskip
\noindent {\bf Type $\mathbf{A_{n-1}}$}.

The only elliptic conjugacy class is that of the Coxeter elements, and
by Proposition \ref{p:cox},  $\sigma(\Cox)^{o(\Cox)}=z_G$.

Now consider the twisted case, so $\delta$ is the non-trivial distinguished involution. 
Suppose $(a_1,\dots, a_l)$ is a partition of $n$ with all odd parts. 
Let $w=(\Cox_{a_1}\times\dots\times \Cox_{a_l})\epsilon\in W\delta$
(see  \eqref{e:EP}).
Set $d=\lcm(a_1,\dots, a_l)$.
Since $d$ is odd it is easy to see that $w^d=\epsilon$, and $o(w)=2d$. 
Choose a representative $g\in G$ of $\Cox_{a_1}\times\dots\times \Cox_{a_l}$.
Without loss of generality we may assume $\gti=g$. 
By Proposition \ref{p:cox} (applied to each factor)  $g^{2d}=I$.
Then
$$
\begin{aligned}
(g\epsilon)^{2d}&=(g\epsilon g\epsilon^{-1}\epsilon^2)^d\\
&=(g(\gti)(-I)^{n+1})^d\\
&=g^{2d}(-I)^{d(n+1)}\\
&=(-I)^{n+1}\\
&=z_G.
\end{aligned}
$$
Note that this is independent of $w$, and 
$\sigma(w)^{o(w)}=1$ if and only if $\ch\rho\in X_*$.

\medskip
\noindent {\bf Type $\mathbf{C_n}$}.
Suppose $\P=(a_1,\dots, a_l)$ is a partition of $n$
and $w$ is in the corresponding elliptic conjugacy class
$\E(\P)$  (cf. Section \ref{s:ellclassical}).
Set $e=\lcm(a_1,\dots, a_l)$. 
Since $o(\Cox_n)=2n$, and $Sp(2a_1)\times \dots\times Sp(2a_l)$ embeds in $Sp(2n)$, it is easy to see that 
$$
o(w)=2e.
$$
Recall (Proposition \ref{p:cox}) $\sigma(\Cox_n)^{2n}=(2\ch\rho)(-1)$. It follows easily that if we set
$$
\ch\tau=
(\frac{e}{a_1}\ch\rho(C_{a_1})\times \dots\times \frac{e}{a_l}\ch\rho(C_{a_l}))
$$
then it follows that 
$$
\sigma(w)^{o(w)}=(2\ch\tau)(-1).
$$
and this is trivial if and only if 
$\ch\tau\in X_*(T)$.
At least one term $e/a_i$ is odd. It follows that if $G$ is simply connected then $\ch\tau\not\in X_*(T)$,
and if $G$ is adjoint this holds 
if and only if all $e/a_i$ are odd, or equivalently
if and only if each $a_i$ has the same power of $2$ in its prime decomposition.

\medskip

\noindent {\bf Type $\mathbf{B_n/D_n}$}. 
Suppose $\P=(a_1,\dots, a_l)$ is a partition of $n$ 
with $a_1\ge a_2\ge \cdots\ge a_l\ge 1$, and $w$ is an element of the corresponding 
elliptic conjugacy class $\E(\P)$ (cf. Section \ref{s:ellclassical}). 
Then  $o(w)=2\lcm(a_1,\dots, a_l)$.
For $1\le i\le l$ set 
$$
\begin{aligned}
\Sigma(\P,i)&=\sum_{k=0}^{i-1}a_k\\
[a,b]&=\{a,a+1,\dots, b\}\\
S_i&=
\begin{cases}
[\Sigma(\P,i)+1,n]&\text{type $B_n$}\\
[\Sigma(\P,i)+1,n]&\text{type $D_n$}, \Sigma(\P,i)\le n-2\\
\emptyset&\text{type $D_n$}, \Sigma(\P,i)> n-2
\end{cases}\\
e_i&=2o(w)/a_i\in 2\Z.
\end{aligned}
$$
There exists an element $w$ in the corresponding
elliptic conjugacy class with
$$
j(w)^{o(w)}=j(w_0)^{e_1} j(w_0(S_2))^{e_2-e_1} \cdots
j(w_0(S_l))^{e_l-e_{l-1}}.
$$
Set $e_0=0$ and 
$$
\ch\tau=\sum_{i=1}^\ell \frac{e_i-e_{i-1}}2\ch\rho(S_i).
$$
Then $\sigma(w)^{o(w)}=1$ if and only if $\ch\tau\in X_*(T)$. 
This is automatic if $G$ is adjoint or $\SO(2n)$.
\begin{example}
\label{ex:D}
Consider the partition $(2,1,\dots,1)$ of $n\ge 3$.
Then $o(w)=4$, $e_1=2$, $e_2=e_3=\dots e_{n-1}=4$,  $S_1=\Pi$, $S_2=\{3,4,\dots, n\}$.
Since $e_i-e_{i-1}=0$ for $i\ge 2$ we get 
$\ch\tau=\ch\rho+\ch\rho(\{3,\dots, n\})$, i.e. in the standard coordinates
$$
\ch\tau=(n-1,n-2,2(n-3),2(n-4),\dots, 2,0).
$$  
The sum of the coordinates of $\ch\tau$ is odd. Therefore $\ch\tau$ is in $X_*(T)$ if $G$ is adjoint 
or $G=\SO(2n)$, but not if $G=\Spin(2n)$. By Lemma \ref{l:easynonlift} this implies $W$ does not lift to $\Spin(2n)$. 
\end{example}

\bibliographystyle{plain}

\end{document}